\documentclass[a4paper, 11pt]{article}
\usepackage[left=3cm,right=2cm,top=2.5cm,bottom=2.5cm]{geometry}

\usepackage[utf8]{inputenc}
\usepackage{lmodern,stackrel}
\usepackage[T1]{fontenc}
\usepackage{verbatim}
\usepackage{soul}
\usepackage{textcomp}
\usepackage[english]{babel}
\usepackage[pdftex]{hyperref}
\usepackage[pdftex]{color,graphicx}
\usepackage{multicol}
\usepackage{lipsum}                    
\usepackage{xargs}                     
\usepackage[pdftex,dvipsnames]{xcolor}

\usepackage{amsmath, mathtools, amsfonts, amssymb, amsthm, mathrsfs, stackrel}
\usepackage{upgreek, graphicx, epstopdf, dsfont, soul}
\usepackage{tikz,xcolor}
    \usetikzlibrary{shapes,arrows}
    \usetikzlibrary{arrows,calc,positioning}
    \usepackage{pgfplots}
    \tikzset{
        block/.style = {draw, rectangle,
            minimum height=1cm,
            minimum width=1.5cm},
        input/.style = {coordinate,node distance=1cm},
        output/.style = {coordinate,node distance=4cm},
        arrow/.style={draw, -latex,node distance=2cm},
        pinstyle/.style = {pin edge={latex-, black,node distance=2cm}},
        sum/.style = {draw, circle, node distance=1cm},
    }

\renewcommand{\leq}{\leqslant}
\renewcommand{\geq}{\geqslant}

\def\build#1_#2^#3{\mathrel{
\mathop{\kern 0pt#1}\limits_{#2}^{#3}}}

\newcommand{\Pp}[1]{\mathbb{P}_p\left(#1\right)}
\renewcommand{\P}[1]{\mathbb{P}_p\left(#1\right)}
\newcommand{\edge}[1]{\left\langle #1\right\rangle}
\newcommand{\Z}{{\mathbb{Z}}}

\newcommand{\R}{\mathbb{R}}
\newcommand{\conex}{\longleftrightarrow}
\newcommand{\norm}[1]{\left\|#1\right\|}

\theoremstyle{plain}
\newtheorem{theorem}{Theorem}
\newtheorem{corollary}{Corollary}

\newtheorem{proposition}[corollary]{Proposition}

\newtheorem{lemma}{Lemma}

\theoremstyle{definition}

\newtheorem{remark}{Remark}

\begin{document}

\title{The probability of connection between two vertices cannot be monotone with the distance for Bernoulli Percolation on transitive graphs}

\author{Alberto M. Campos\footnote{Departamento de Matemática, Universidade Federal de Minas Gerais, Av. Antônio Carlos 6627,  30123-970, Belo Horizonte-MG, Brazil}\; and Bernardo N.B. de Lima\footnote{Departamento de Matemática, Universidade Federal de Minas Gerais, Av. Antônio Carlos 6627,  30123-970, Belo Horizonte-MG, Brazil}}
\date{}
\maketitle


\begin{abstract}
 A popular question in Bernoulli percolation models is if the probability of connection between two vertices in a transitive graph decays monotonically with the distance between these two vertices. For example, on the square lattice is an open question to prove that the probability of the origin being connected to the vertex $(0,n)$ is monotone in $n$. In this short note, we exhibit an example of a  transitive graph in which the probability of connection between vertices does not necessarily decay as the distance of those vertices grows. We also define a critical point for percolation in $\mathbb{Z}^d$, in which using a generalization of the percolation process it is possible to see the same phenomena happening in the embedding of $\mathbb{Z}^d$ over $\mathbb{R}^d$.  
\end{abstract}

{\footnotesize Keywords:  Percolation theory; Percolation;   \\
MSC numbers:  60K35, 82B43}
\section{Introduction}

Bernoulli Percolation is a source of several questions very easy to state but difficult to resolve, {including} open questions. One such problem is to demonstrate that the probability of connection between vertices is monotone with respect {to} the graph distance. Essentially, this question asks whether, in a random medium, the probability of two points being connected within a cluster decreases as the distance between these points increases. This problem was first proposed in 1965 by Hammersley and Welsh \cite{Hammersley1965} in the context of First Passage Percolation. Since then, numerous papers addressing this topic have been published, including \cite{Ahlberg2015,Alm1999,Berg1983,Baptiste2012,Howard2001} in the context of First Passage Percolation, and \cite{Lima2015} in Percolation Theory. A similar and equally captivating question was explored in \cite{Andjel2008} within the framework of oriented percolation. Additionally, it is worth mentioning the bunkbed conjecture, which poses a related question in the context of finite slabs; see \cite{Nikita2025}.


To formalize the problem, let $G=(\mathbb{V}(G),\mathbb{E}(G))$ be a graph, where $\mathbb{V}(G)$ is the set of vertices and $\mathbb{E}(G)\subseteq \{\{v,u\} : v,u\in \mathbb{V}(G), v\neq u\}$ is the set of non-oriented edges. For every $v,u\in \mathbb{V}(G)$, define a path in $G$ from $v$ to $u$ to be an alternating sequence of vertices and edges of the form $\gamma=(v=v_0,e_1,v_1,e_2,\ldots,e_n,v_n=u)$, where $v_i,v_{i+1}\in e_i$ for every $i\in\{0,\ldots,n-1\}$. Let $|\gamma|$ denote the number of edges in the path $\gamma$. Then, for every $v,u\in \mathbb{V}(G)$, define the graph distance:
\begin{align*}
    d_G(v,u)=\min\{ |\gamma| : \gamma \text{ is a path in $G$ from } v \text{ to } u \}.
\end{align*}
If there is no path connecting $v$ to $u$, set $d_G(v,u)=\infty$. A graph $G$ is connected if, for every $v,u\in \mathbb{V}(G)$, we have $d_G(v,u)<\infty$.

We say that a graph $G$ is transitive if, for every pair of vertices $v,u \in \mathbb{V}(G)$, there exists a graph isomorphism that maps the vertex $v$ to $u$. In a transitive graph $G$, we can fix any vertex to be the origin and denote it by $o$. With the origin fixed, define the norm of each $v \in \mathbb{V}(G)$ as $\|v\|_G = d_G(o,v)$.

Consider a Bernoulli percolation model with parameter $p$ on the graph $G$ with the standard underlying probability space denoted by $(\Omega, \mathcal{F}, \mathbb{P}_p)$.

Given $v, u \in \mathbb{V}$, as usual in percolation, define the event $\{v \conex u\}$ as the set of configurations in $\Omega$ where there exists a path formed by open edges connecting $v$ and $u$. Moreover, define the percolation event $\{v \conex \infty\}$ as the set of configurations where there exists an infinite open path starting from $v$.

We say that a transitive graph $G$ has the \textbf{\textit{monotonic property}} at parameter $p \in [0,1]$, if for all pairs of vertices $x, y \in \mathbb{V}(G)$ with $\|x\|_G < \|y\|_G$, it holds that $\mathbb{P}_p(o \conex x) \geq \mathbb{P}_p(o \conex y)$.


Given a positive integer $d$, let $\mathbb{Z}^d$ be the $d$-dimensional hypercubic lattice. The problem of determining whether the graph $\mathbb{Z}^d$ satisfies the monotonic property remains open for all $d\geq 2$. However, the intuitive aspect of this problem has already revealed impressive results. One of the most intriguing contributions was given by J. van den Berg \cite{Berg1983}, where he provided a negative answer on the graph $\mathbb{Z} \times \mathbb{Z}^+ \subset \mathbb{Z}^2$ in the context of First Passage Percolation.

In Section \ref{Sec:2}, we will demonstrate that transitivity alone is not sufficient for a graph to exhibit the monotone property. Specifically, we will prove the following statement.
\begin{theorem}
\label{teo:conterexample}
For every $\beta \in (0,1)$, there exists a transitive graph $G$ and vertices $x, y \in \mathbb{V}(G)$ such that $\|x\|_G < \|y\|_G$, but $\mathbb{P}_p(o \conex x) < \mathbb{P}_p(o \conex y)$ for every $p \in [\beta,1)$.
\end{theorem}

Carrying out a careful analysis of the proof of Theorem \ref{teo:conterexample}, it is possible to show, in a new model of percolation that generalizes the usual Bernoulli percolation, that even hypercubic lattices can fail to exhibit the monotonic property. 




The \textit{\textbf{Pipe-Dust percolation model}} on $\mathbb{Z}^d$ with parameter $\lambda > 0$ is defined as follows. {Given $e=(x,y)\in \mathbb{E}(\mathbb{Z}^d)$, let $I_e=\{(x,x+t(y-x))\in \R^d: t\in [0,1]\}$ }be the line segment joining the end-vertices of $e$. Define an independent Poisson point process with rate $\lambda$ on each segment $(I_e)_{e\in \mathbb{E}}$. Let us denote by $\mathbb{P}_{\lambda}$ the underlying product measure over all these Poisson processes. Each point of the Poisson process is called~\emph{dust}.

In this model, each edge of $\mathbb{Z}^d$ contains no dust with probability $e^{-\lambda}$. In this case, the edge is said to be completely open. Otherwise, if the interval contains one or more dust particles (resembling dust inside a pipe), each dust acts as an obstruction. Given two points $x,y \in \cup_{e\in\mathbb{E}} I_e$, we say that $x$ and $y$ are connected if there exists a continuous path contained in $\cup_{e\in\mathbb{E}} I_e$ connecting $x$ to $y$ and avoiding all dust particles.

The Pipe-Dust percolation model generalizes the classical Bernoulli percolation with parameter $p$, taking $\lambda = -\ln(p)$. Moreover, this model allows us to investigate connections between non-integer points.



We consider this model in Section~\ref{Sec:3} and prove the following theorem:

\begin{theorem}\label{teo:2}
For every $d \geq 2$, there exists a critical value $\lambda_c(\mathbb{Z}^d) \in (0,\infty)$ such that
\begin{enumerate}
    \item[a)] If $\lambda > \lambda_c$, then $\mathbb{P}_{\lambda}\left(o\conex \vec{e}\, \right) > \mathbb{P}_{\lambda}\left(o\conex t\vec{e}\, \right)$ for all $t\in (0,1)$.
    \item[b)]\label{teo:2b} If $0 < \lambda < \lambda_c$, then there exists
    $t = t(\lambda) \in (0,1)$ such that $\mathbb{P}_{\lambda}\left(o\conex \vec{e}\, \right) < \mathbb{P}_{\lambda}\left(o\conex t\vec{e}\, \right)$.
\end{enumerate}
Where $\vec e = (1,0,\ldots,0)\in\mathbb{Z}^d$.
\end{theorem}

To illustrate this theorem, we present the Remark~\ref{ex:1}, where it is shown that $\lambda_c(\mathbb{Z}^2) = -\ln(p_c(\mathbb{Z}^2))$. In contrast, Remark~\ref{ex:2} illustrates that the same behavior does not occur in the triangular lattice.


\textbf{Acknowledgement:} The authors thank Augusto Teixeira for useful discussions. The research of A.M.C. and B.N.B.L. is supported by CNPq grants 315656/2025-5 and 315861/2023-1, respectively.

\section{Proof of Theorem \ref{teo:conterexample}}\label{Sec:2}

Before proving Theorem \ref{teo:conterexample} in its full generality, we present a simple construction divided into two parts. The first part explores the properties of connections in a finite graph $P$. The second part connects graphs isomorphic to $P$ in a tree-like structure to construct a transitive graph that does not satisfy the monotonic property. The proof itself will follow the same steps as this construction but starts with a more general graph $P$.

\begin{figure}[!ht]
    \centering
    \includegraphics[]{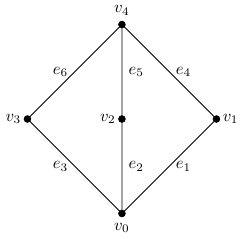}
    \caption{Representation of the graph $P$, and elements $e_i$, $i\in \{1,...,6\}$ responsible for assembling the final graph $G$.}
    \label{fig:1}
\end{figure}

Let us start by defining the graph $P=(\mathbb{V}_P,\mathbb{E}_P)$, where $\mathbb{V}_{P}=\{v_0,v_1,v_2,v_3,v_4\}$ and $\mathbb{E}_P=\{\edge{v_0,v_1},\edge{v_1,v_4},\edge{v_4,v_3},\edge{v_3,v_0},\edge{v_0,v_2},\edge{v_2,v_4}\})$. In this graph, as depicted in Figure \ref{fig:1}, define the set of \textit{\textbf{middle vertices}} as $\{v_1,v_2,v_3\}$ and the set of \textit{\textbf{peak vertices}} as $\{v_0, v_4\}$, where $v_0$ will be identified as the origin.  Note that $\norm{v_4}_P=2>\norm{v_1}_P=1$, and in the standard Bernoulli percolation, the probability of connection is given by:
\begin{align*}
    \Pp{v_o\conex v_4}&=1-(1-p^2)^3,\\
    \Pp{v_o\conex v_1}&=p+(1-p)(1-(1-p^2)^2)p.
\end{align*}Taking the difference of those polynomials, we get:
\begin{align*}
    \Pp{v_o\conex v_1}-\Pp{v_o\conex v_4}&=p(p-1)^2\left(p+\frac{1-\sqrt{5}}{2}\right)\left(p+\frac{1+\sqrt{5}}{2}\right),
\end{align*}
then for $\frac{\sqrt{5}-1}{2}<p<1$, it follows:
\begin{align*}
    \Pp{v_o\conex v_4}>\Pp{v_o\conex v_1}.
\end{align*}

To create a transitive graph $G$ using $P$, we define the gluing operation. Given two graphs $H_1$ and $H_2$, and two fixed vertices $x \in \mathbb{V}(H_1)$ and $y \in \mathbb{V}(H_2)$, the gluing operation produces a new graph $H_3$ by removing the vertices $x$ and $y$ from $H_1$ and $H_2$, and introducing a new vertex $z$. The vertex $z$ inherits all edges previously incident to $x$ and $y$, meaning that $z$ becomes adjacent to all vertices that were originally adjacent to either $x$ or $y$.

Consider an infinite stack of graphs $\{H_j\}_j$, where $H_j = P$ for every $j$, and define the graph $G$ inductively. Start by setting $G_0 = P$, and remove five graphs from the stack, one for each non-glued vertex of $G_0$. For each vertex of $G_0$, glue it to one of the removed graphs from the stack. Specifically, if the vertex is a peak vertex in $G_0$, glue it to a middle vertex, or if the vertex is a middle vertex in $G_0$, glue it to a peak vertex. This construction results in a graph $G_1$ with twenty non-glued vertices. For the general $k$-th step, starting with the graph $G_{k-1}$, remove $5 \cdot 4^{k-1}$ graphs from the stack, one for each non-glued vertex of $G_{k-1}$. Then, for each non-glued vertex of $G_{k-1}$, if the vertex is a peak vertex, glue it to a middle vertex, or if the vertex is a middle vertex, glue it to a peak vertex; this defines the graph $G_{k}$. Since $\{G_k\}_k$ forms an increasing sequence of graphs, define $G$ to be its limit. A representation of a piece of $G$ is shown in Figure \ref{fig:2}.

\begin{figure}[!ht]
    \centering
    \includegraphics[]{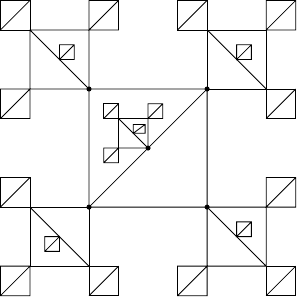}
    \caption{A representation of $G_2$, with copies of graph $P$ glued together in a tree-like structure, representing a piece of graph $G$.}
    \label{fig:2}
\end{figure}

Note that $G$ is transitive, but by construction, it fails to exhibit the monotone property, for all $p > \frac{\sqrt{5} - 1}{2}$.

 The proof of Theorem \ref{teo:conterexample} follows a very similar argument. Begin by defining the graph $P_n = (V_n, E_n)$, where $V_n=\{v_0,v_1,...,v_n\}$, and $E_n=\{\edge{v_j,v_i} ; i\in \{1,...,n-1\},j\in\{0,n\}\}$. Additionally, denote $\{v_0,v_n\}$ to be the peak vertices, and $\{v_1,...,v_{n-1}\}$ to be the middle vertices. Then, related to connection probabilities, we have
\begin{align*}
    \Pp{v_o\conex v_i}=
    \begin{cases}
    1-(1-p^2)^{n-1} &, \; i=n;\\
    p+(1-p)p(1-(1-p^2)^{n-2}) &,\; i\in \{1,2,...,n-1\}.
    \end{cases}
\end{align*} Therefore,
\begin{align*}
    \Pp{v_o\conex v_n}-\Pp{v_o\conex v_1}&=(1-p)(1-p-(1-p^2)^{n-2}),\\
    &\geq (1-p)^2(1-2(1-p^2)^{n-3}).
\end{align*} This implies that the difference in probabilities is positive if $p\geq(1-2^{\frac{-1}{n-3}})^{\frac{1}{2}}$. Furthermore, for every $\beta\in(0,1)$, there exists $n(\beta)>0$ such that for every $n>n(\beta)$, it is true that: 
\begin{align*}
    \Pp{v_o\conex v_n}> \Pp{v_o\conex v_i},\forall p\in [\beta,1) \text{ and }i\in \{1,...,n-1\}.
\end{align*}Again, note that  $\norm{v_n}_{P_n}=2>\norm{v_i}_{P_n}=1$, for all $i\in \{1,...,n-1\}$.

To construct the non-transitive graph $G^{(n)}$, follow the same construction as the graph $G$, but instead of using a stack of graphs $P$, use a stack of graphs $P_n$. In particular, for every $p \in [ \beta, 1)$, by taking $n > n(\beta)$, the graph $G^{(n)}$ does not satisfy the monotone property, concluding the proof.

\section{The Dust-Pipe percolation}\label{Sec:3}\noindent

The construction in Section \ref{Sec:2} demonstrates that there exist artificial transitive graphs that do not satisfy the monotonic property. However, it remains an open question whether graphs like $\mathbb{Z}^d$ satisfy the monotonic property or not. 

Using the Dust-Pipe percolation model on $\mathbb{Z}^d$, one can find points (not necessarily with integer coordinates) that serve as counterexamples for the monotone property at certain values of $\lambda$. To illustrate that the model exhibits a threshold related to this property, observe Figure \ref{fig:3}, and consider the $2d-1$ paths of length less than three that connect the origin $o$ to the point $\vec{e} = (1, 0, \dots, 0)$. Note that, similar to Section~\ref{Sec:2},   increasing the number of parallel paths also increases the difference between the connection probability of the peak and middle vertices. 

\begin{figure}[!ht]
    \centering
    \includegraphics[]{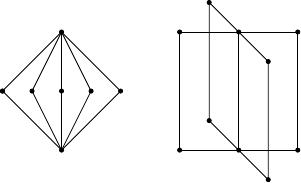}
    \caption{Representation of the graph $P_6$, and the closest paths between $o$ and $(1,0,0)$.}
    \label{fig:3}
\end{figure}

For the Dust-Pipe model, increasing the dimension can force a point in the segment $(o, \vec{e}\,)$ to have a lower connection probability compared to the vertex $\vec{e}$. The critical dimension where this phenomenon begins to occur is $d = 2$.

\begin{theorem}
\label{teo:existence} For every $d\geq 2$, there 
exists $\lambda_0\in (0,\infty)$ and points $x,y \in \bigcup_{e\in \mathbb{E}(\mathbb{Z}^d)}I_e$, such that $\norm{x}<\norm{y}$ and $\mathbb{P}_{\lambda}(o\conex x)<\mathbb{P}_{\lambda}(o\conex y)$ for every $\lambda\in (0,\lambda_0)$.
\end{theorem}
\begin{proof}
 Fix $\vec{e}=(1,0,\ldots,0)\in \mathbb{V}(\mathbb{Z}^d)$, and take $t\in(0,1)$. Then, 
\begin{align} \label{eq:existencebound}
    \mathbb{P}_{\lambda}(o\conex t\vec{e}\,)&\leq e^{-\lambda t}+e^{-\lambda(1-t)}-e^{-\lambda},\\
    \mathbb{P}_{\lambda}(o\conex \vec{e}\,)&\geq 1-(1-e^{-3\lambda})^{2(d-1)}(1-e^{-\lambda}).\nonumber
\end{align} The first inequality relies on the assumption that $\mathbb{P}_{\lambda}(o \conex \vec{e}\,) \leq  1$, while the second restricts the space to paths that consist of three edges or fewer. Figure \ref{fig:3} provides an illustration of such paths in $\mathbb{Z}^3$.

Observe that the bound in Equation~\eqref{eq:existencebound} remains the same for every $d$, and the probability $\mathbb{P}_{\lambda}(o \conex \vec{e}\,)$ increases with $d$. Therefore, if Theorem \ref{teo:existence} can be proven in dimension $d = 2$ using the bounds in equation \eqref{eq:existencebound}, it also holds true in higher dimensions. 

Hence, set $d = 2$, $t = \frac{1}{2}$ and $z = e^{-\lambda/2}$. Comparing the points $(1,0)$ and $\left(\frac{1}{2}, 0\right)$, we obtain:
\begin{align*}
    \mathbb{P}_{\lambda}\left(o\conex \left(\frac{1}{2},0\right)\right)&\leq z(2-z),\\
    \mathbb{P}_{\lambda}(o\conex (1,0))&\geq 1-(1-z^6)^{2}(1-z^2).
\end{align*}

Thus, taking the difference between the bounds above, it holds
\[1-(1-z^6)^{2}(1-z^2)-z(2-z)=z(z-1)^2g(z),\]
where $g(z)=z^{11}+2(z^{10}+z^9+z^8+z^7+z^6-z^4-z^3-z^2-z-1)$.

Since $g(1) = 1$ and $g$ is continuous, there exists $z_0<1$, such that $g(z)>0$ for all $z\in (z_0, 1]$. Thus,
\[\mathbb{P}_{\lambda}\left(o\conex \left(\frac{1}{2},0\right)\right) \leq z(2-z) < 1-(1-z^6)^{2}(1-z^2)\leq  \mathbb{P}_{\lambda}(o\conex (1,0)),\]
as desired. To exemplify, we can take $z_0=0.99$, or $\lambda<0.02$.
\end{proof}

Theorem \ref{teo:existence} shows that for $d \geq 2$, when $\lambda$ is close to zero, the connection probability near the origin is not monotone with respect to the distance. In particular, the topology of the set of points with high connection probability undergoes a phase transition, transitioning between being a connected set or not. Therefore, define:

\begin{equation}\label{eq:criticalcontinuouspoint}
\lambda_c(\mathbb{Z}^d)=\sup\left\{\lambda>0 :\begin{aligned}  \text{There exists}\, t\in& (0,1] \text{ such that }\\  \mathbb{P}_{\lambda}\left(o\conex \vec{e}\, \right)>&   \mathbb{P}_{\lambda}(o\conex t\vec{e}\,)\,\end{aligned}\right\}.
\end{equation}



To continue, it is useful to describe this phenomenon using the Bernoulli percolation framework. Before, we need the following lemma:

\begin{lemma}\label{Lem:characterization}
Let $a,b,c\in[0,1]$ and $\lambda\in (0,\infty)$ arbitrary, define the function:
\begin{align*}
    f_{\lambda}(t)=ae^{-t\lambda}+be^{-(1-t)\lambda}-ce^{-\lambda},\ \forall t\in[0,1].
\end{align*} 
Then, it holds
\begin{align*}
    \min_{t\in[0,1]}f_{\lambda}(t)=\begin{cases}
    f_{\lambda}(1) &\text{, if } \frac{b}{a}\leq  e^{-\lambda};\\
    f_{\lambda}(0) &\text{, if } \frac{b}{a}\geq e^{\lambda};\\
    f_{\lambda}(t_0) &\text{, for some } t_0\in(0,1)\text{ if } e^{-\lambda}< \frac{b}{a}<  e^{\lambda}.
    \end{cases}
\end{align*}
\end{lemma}
\begin{proof}
Taking the derivative, 
\begin{align*}
    f_{\lambda}'(t)=0&\iff t=\frac{1}{2}-\frac{1}{2\lambda}\ln\left(\frac{b}{a}\right).
\end{align*} If $b/a\leq e^{-\lambda}$, then $t\geq 1$, therefore the minimum is attained in $t=1$. An analogous result holds when $b/a\geq e^{\lambda}$. If $e^{\lambda}>b/a>e^{-\lambda}$, then there exists a minimum with some $t\in(0,1)$.
\end{proof}

Observing that 
\begin{align*}
    \mathbb{P}_{\lambda}\left(o\conex t \vec{e}\,\right)=e^{-t\lambda}+  P_{\lambda}\left(o\conex \vec{e}\,\middle| \edge{o,\vec{e}\,}\text{ is closed}\right)(e^{-(1-t)\lambda}-e^{-\lambda}). 
\end{align*} By Lemma \ref{Lem:characterization}, whenever $P_{\lambda}\left(o\conex \vec{e}\,\middle| \edge{o,\vec{e}\,}\text{ is closed}\right)> e^{-\lambda}$, the probability of connection is not monotone near the origin. 
    
    
Now, considering ordinary Bernoulli percolation with parameter $p$, we define the threshold:
\begin{align}
\label{eq:criticalpoint}
\tau_c(\mathbb{Z}^d)=\inf\left\{p \in [0,1] \,: \mathbb{P}_p(o\conex \vec{e}\,\middle| \edge{o,\vec{e}\,}\text{ is closed})>p\right\}.
\end{align} And, with the canonical relation one has $\lambda_c(\mathbb{Z}^d)=-\ln{\tau_c(\mathbb{Z}^d)}$. 

The majority of references in percolation theory utilize the traditional Bernoulli notation. Therefore, working with the definition of $\tau_c(\mathbb{Z}^d)$ in some cases is more natural than $\lambda_c(\Z^d)$. It is essential to emphasize that, beyond the definition, the phenomenon of monotonicity in connection probabilities is only proved in the Dust-Pipe model and it is open in the Bernoulli percolation model.

 Now, let's prove that those values are not trivial, i.e. $\tau_c(\mathbb{Z}^d)\in (0,1)$ or $\lambda_c(\mathbb{Z}^d)\in(0,\infty)$.

\begin{theorem}\label{teo:nor0nor1}
For $d\geq 2$, there are $0<p_0<p_1<1$ such that for every $p\in(0,p_0)$, we have that $\mathbb{P}_p\left(o\conex \vec{e}\,\middle| \edge{o,\vec{e}\,}\text{ is closed }\right)<p$; and for every $p\in(p_1,1)$, we have that $$\mathbb{P}_p\left(o\conex \vec{e}\,\middle| \edge{o,\vec{e}\,}\text{ is closed }\right)>p.$$
\end{theorem}
\begin{proof}
Comparing the probability of connection with the { probability of the existence of a path }of length $3$ (not necessarily ending in the vertex $\vec{e}\,$), we get:
\[\P{o\conex \vec{e}\,\middle| \edge{o,\vec{e}\,} \text{closed} }<\P{\bigcup_{\gamma:|\gamma|=3} \{\text{The path }\gamma \text{ is open}\}}\leq 2d(2d-1)^2 p^3.\]
Thus, when $p<(2d(2d-1))^{-1/2}$, the probability is less than $p$, and we find $p_0$. Using Theorem~\ref{teo:existence}, we can take $p_1=0.99\geq e^{-0.02}$, and that concludes the Theorem for all dimensions $d\geq2$.
\end{proof}

Since $\tau_c(\mathbb{Z}^d)$ belongs to the closed interval $[p_0, p_1] \subset (0,1)$, it remains to show that indeed $\tau_c$ represents a critical phenomena. Essentially, below $\tau_c$, the connection probability is monotone, while above $\tau_c$, the connection probability is not. To establish this, consider the following theorem:

\begin{theorem}\label{thm:analytic}
The function $F(p)=\mathbb{P}_p\left(o\conex \vec{e}\,\middle| \edge{o,\vec{e}\,}\text{ is closed }\right)$ is an analytic function for $p\in[0,1]\setminus \{p_c(\mathbb{Z}^d)\}$.  Moreover, for every $p\in (\tau_c(\mathbb{Z}^d),1)$, it holds that $\P{o\conex \vec{e}\,\middle| \edge{o,\vec{e}\,} \text{closed} }>p$. 
\end{theorem}

\begin{proof} 
The proof that, for all $d\geq 3$, $F(p)$ is an analytic function in $[0,1] \setminus \{p_c(\mathbb{Z}^d)\}$ follows the arguments of three theorems from Grimmett's book \cite{Grimmett1999}. Start by applying the animal construction described in Theorem (6.108), and when $p < p_c$, use the exponential decay provided by Theorem (5.4) to ensure that the series is analytic. Then, for $p > p_c$, use the exponential decay of the two-arm probability, stated in Lemma (7.89). 

With the analyticity of $F(p)$ proven, we are going to show that the equation $F(p)=p$ have a unique non trivial solution with $p\in (0,1)$. For this consider the following proposition:

\begin{proposition}[Theorem (2.38) of \cite{Grimmett1999}]\label{thm:logP}
Let $A$ be a increasing event which depends on only finitely many edges of $\mathbb{Z}^d$, and suppose that $0<p<1$. Then $\log{\Pp{A}}/\log{p}$ is a non-increasing function of $p$.
\end{proposition}

Let $n > 0$, and set $\Lambda_n=\{x\in \Z^d: \|x\|\leq n\}$ as a the discrete box of size $n$. Define $B_n$ as the box $\Lambda_n$ with the edge $\edge{ o, \vec{e}\,}$ removed beforehand. The event $\{0 \overset{B_n}{\conex} \vec{e}\,\}$ is an increasing event that depends on a finite number of edges, and thus the function
\begin{align*}
    h_n(p) = \frac{\log \mathbb{P}_p(0 \overset{B_n}{\conex} \vec{e}\,)}{\log p}
\end{align*}
is non-increasing by Proposition~\ref{thm:logP}.

Since $h_n$ is a non-increasing continuous function, for every $n \geq 2$ Theorem~\ref{teo:nor0nor1} implies that $h_n(p) {>} 1$ for $p < p_0$ and $h_n(p) {<} 1$ for $p > p_1$. Consequently, there exists a non-trivial value $p^{\ast}$ such that $h_n(p^{\ast}) = 1$.

If there exist two points $p_1^{\ast}<p_2^{\ast}\in(0,1)$ such that $h_n(p_1^{\ast}) = h_n(p_2^{\ast})=1$, then by Proposition~\ref{thm:logP} the probability $\mathbb{P}_p(0 \overset{B_n}{\conex} \vec{e}\,)$ must be equal to $p$ in $(p_1^{\ast},p_2^{\ast})$. However, this leads to a contradiction, as $\mathbb{P}_p(0 \overset{B_n}{\conex} \vec{e}\,)$ is a polynomial with a degree greater than one.

Moreover, if there exist points $p_1^{\ast}< p_2^{\ast}\in (0,1)$ such that the limits of $h_n$ satisfy
\begin{align*}
    \lim_{n \to \infty} h_n(p_1^{\ast}) = \lim_{n \to \infty} h_n(p_2^{\ast})=1,
\end{align*}
then for every point $p\in(p_1^{\ast},p_2^{\ast})$, the analytic probability $F(p)$ would be equal to $p$. This is also a contradiction, as near $p = 0$ and $p = 1$, this equality does not hold.

This implies that $F(p) = p$ has at most one unique root for $p \in (0,1)$. Thus, by the non-increasing argument, it follows that for every $p \in (\tau_c(\mathbb{Z}^d), 1)$, the probability of connection in the space, conditioned on the event $\{\edge{o, \vec{e} \,}\text{ closed}\}$, is greater than $p$, concluding the proof.
\end{proof}

Since $\lambda_c(\mathbb{Z}^d) = -\ln{\tau_c(\mathbb{Z}^d)}$, it follows that $\lambda_c(\mathbb{Z}^d)$ is also a well-defined, non-trivial critical point. Therefore, Theorem~\ref{teo:2}, item~$b)$ is a direct consequence of Theorem~\ref{thm:analytic}, while item~$a)$ follows from the definition~\eqref{eq:criticalpoint}, Theorem~\ref{teo:nor0nor1}, and Theorem~\ref{thm:analytic}. We conclude this paper with the following remarks: 

\begin{remark} \label{ex:1}
 Observe that at the critical point of $\mathbb{Z}^2$, the dual lattice can be used to show that
for $\mathbb{Z}^2$, we have that $\tau_c(\mathbb{Z}^2)=\frac{1}{2}=p_c(\mathbb{Z}^2)$. 
\end{remark}

\begin{remark}
Using FKG inequality and the characterization Lemma \ref{Lem:characterization}, for $p>\tau_c$, the connection probability attains a local minimum in the interior of the edge for every edge.
\end{remark}
 
\begin{remark} \label{ex:2}
Let $\mathrm{T}$ and $\mathrm{H}$ denote the triangular and hexagonal lattices, respectively, and let $\vec{e}_{\mathrm{T}}$ and $\vec{e}_{\mathrm{H}}$ be any nearest-neighbor vertices of the origin $o$ in the triangular and hexagonal lattices, respectively. Using an argument based on the dual and the substitution method [(11.116) of \cite{Grimmett1999}, page 333], for the critical point $p=2\sin{\pi/18}$:
\begin{align*}
    \begin{cases}
    \mathbb{P}^{\mathrm{T}}_p\left(o\conex \vec{e}_{\mathrm{T}}|  \edge{o,\vec{e}_{\mathrm{T}}\,}  \text{ closed}\right)=p^2 + 2 p (1 - p) A + (1 - p)^2 B;\\
    \mathbb{P}^{\mathrm{H}}_{1-p}\left(o\conex \vec{e}_{\mathrm{H}}|  \edge{o,\vec{e}_{\mathrm{H}}\,}  \text{ closed}\right)=(1 - p)^2 A+ 2 p (1 - p) B;\\
     \mathbb{P}^{\mathrm{H}}_{1-p}\left(o\conex \vec{e}_{\mathrm{H}}|  \edge{o,\vec{e}_{\mathrm{H}}\,}  \text{ closed}\right)+  \mathbb{P}^{\mathrm{T}}_p\left(o\conex \vec{e}_{\mathrm{T}}|  \edge{o,\vec{e}_{\mathrm{T}}\,}  \text{ closed}\right)=1.
    \end{cases}
\end{align*} where $\mathbb{P}^T$ and $\mathbb{P}^H$ are the measures on the triangular and hexagonal lattice. And, for $\{x,y,z\}$  vertices of a triangle in $\mathrm{T}$, let: 
\begin{align*}
        A&= \mathbb{P}^{\mathrm{T}}_p\left(\{x\conex y\} \cup \{x\conex z\}|  \edge{x,y},\edge{y,z} \text{ and } \edge{z,x}  \text{ closed}\right),\\
        B&= \mathbb{P}^{\mathrm{T}}_p\left(\{x\conex y\}|  \edge{x,y},\edge{y,z} \text{ and } \edge{z,x}  \text{ closed}\right).
\end{align*} The solution of the system gives $A=(1-B)$, that implies:
\begin{align*}
    \tau_c(\mathrm{T})<p_c(\mathrm{T})\\
    \tau_c(\mathrm{H})>p_c(\mathrm{H}).
\end{align*} In particular, the $\tau_c$ is not necessarily equal to the critical point of the space. 
\end{remark}

\end{document}